\documentclass[12pt]{iopart}

\usepackage{amsmath}

\usepackage{dsfont}
\usepackage{mathrsfs}
\usepackage{lmodern}
\usepackage[T1]{fontenc}
\usepackage[latin1]{inputenc}

\usepackage{iopams}

\usepackage[amsmath,thmmarks]{ntheorem}

\newtheorem{theorem}{Theorem}[section]
\newtheorem{lemma}[theorem]{Lemma}

\newtheorem{definition}[theorem]{Definition}
\newtheorem{assumption}[theorem]{Assumption}
\newtheorem{remark}[theorem]{Remark}

\theoremstyle{nonumberbreak}
\theorembodyfont{\normalfont} %
\theoremsymbol{$\Box$}
\newtheorem{proof}{Proof}

\renewcommand{\(}{\left(}
\renewcommand{\)}{\right)}

\newcommand{\<}{\left<}
\renewcommand{\>}{\right>}

\newcommand{\IN}          {\mathds{N}}                                        

\newcommand{\SpX}         {\mathcal{X}}
\newcommand{\SpY}         {\mathcal{Y}}

\newcommand{\VarX}        {x}                                                 
\newcommand{\VarY}        {y}                                                 
\newcommand{\Err}         {\xi}

\newcommand{\Op}          {A}                                                 

\newcommand{\BasX}        {u}                                                 

\newcommand{\Fsing}       {\sigma}                                            

\newcommand{\E}           {\mathbb{E}}
\newcommand{\PP}          {\mathbb{P}}

\newcommand{\argmin}      {\operatorname*{argmin}}

\begin{document}

\title{Parameter Choice by Fast Balancing}
\author{Frank Bauer}


\eads{frank.bauer.de@gmail.com}

\date{This version: \today}

\begin{abstract}
Choosing the regularization parameter for inverse problems is of major importance for the performance of the regularization method.

We will introduce a fast version of the Lepskij balancing principle and show that it is a valid parameter choice method for Tikhonov regularization both in a deterministic and a stochastic noise regime as long as minor conditions on the solution are fulfilled.
\end{abstract}


\maketitle

\section{Introduction}

In the following we will consider linear inverse problems \cite{Engl/Hanke/Neubauer:1996,Hofmann:1986}
given as an operator equation
\begin{equation}\label{main}
 \Op \VarX = \VarY,
\end{equation}
where $\Op :\SpX\to \SpY$ is a linear, continuous,
compact operator acting between separable real infinite dimensional Hilbert spaces
$\SpX,\SpY$. Without loss of generality, we assume that $\Op$ has a trivial null-space $N(\Op) = \{ 0 \}$.  Since $\Op$
is compact and $\SpX$ is infinite dimensional, $\Op$ does not have a continuous inverse, which makes
\eqref{main} ill-posed.

For some definitions, but not for the methods themselves, we will need the singular value decomposition of $\Op$. There exist orthonormal bases $(u_k)_{k \in \IN}$ of $\SpX$
and  $(v_k)_{k \in \IN}$ of $\SpY$ and a sequence of decreasing singular values $(\Fsing_k)_{k \in \IN}$ such that
\begin{equation}
 \Op x = \sum_{k=1}^\infty \Fsing_k \<x,u_k\> v_k.
 \end{equation}
Moreover, we assume that the data $\VarY$ are noisy, the noise model for $\Err$ will be specified later.
\begin{equation}\label{eq1}
\VarY^\delta = \Op \VarX + \Err ,\quad\text{$\Err$ noise}.
\end{equation}
In order to counter the ill-posedness, we need to regularize; in this article we will consider only Tikhonov regularization:
\begin{equation}\label{eq2}
 \VarX_n^\delta = \Op_n^{-1} \VarY^\delta := (\Op^* \Op + q_0 q^n)^{-1} \Op \VarY^\delta
\end{equation}
The level $n$ will now be called regularization parameter; $q_0 > 0$ and $0<q<1$ are constants which are discussed later. The noise-free regularized solution is defined as
 \begin{equation}\label{eq3}
\VarX_n = \Op_n^{-1} \VarY = (\Op^* \Op + q_0 q^n)^{-1} \Op \VarY
\end{equation}
The correct choice of the regularization parameter is of major importance for the performance of the method.
In recent times, a number of articles \cite{Goldenshluger/Pereverzev:2000,Mathe/Pereverzev:2003,Bauer/Pereverzev:2005,Bauer/Hohage:2005,Mathe/Pereverzev:2006,Haemarik/Palm/Raus:2007,Bauer/Hohage/Munk:2009} have considered the Lepskij Balancing principle \cite{Lepskij:1990} for choosing this parameter in various situations.

One of the major disadvantages of this parameter choice method in comparison to, for
instance, the Morozov Discrepancy principle (cf. e.g., \cite{Engl/Hanke/Neubauer:1996})
is that one needs to compute all regularized solutions up to a maximal regularization
parameter. On the other hand, this buys stability even in the face of stochastic noise.

We will show that a simplification of the Lepskij balancing principle (now called fast balancing) will yield a valid parameter choice method which performs at least as well as the original. This idea has already been presented in a different form in \cite{Raus/Haemarik:2008}, however in a purely deterministic setting with a focus on convergence results. In contrast our goal is to provide feasible error bounds in realistic situations without requiring the noise level to be near to or converge to $0$. There are methods (see e.g. \cite{Haemarik/Palm/Raus:2008}) which reach by a combination of solutions a better solution, however as the impact to practice has been rather small still, we will keep our focus to choose the best solution in a set of given ones.

The outline of this paper is as follows. First (section \ref{prerequisites}) we will specify the conditions on the solution $\VarX$ and describe two different scenarios of noise, namely a stochastic and a deterministic one. In section \ref{fastbalancing} we will introduce the Lepskij balancing principle and the fast balancing principle. In the two following sections \ref{deterministic} and \ref{stochastic} we will show oracle inequalities for the new method.

\section{Prerequisites}\label{prerequisites}

For Tikhonov regularization it always holds that $\| \VarX - \VarX_{n+1} \|  \leq  \| \VarX - \VarX_{n} \|$. We will need a slightly more powerful inequality at this point and therefore use a set of assumptions introduced in \cite{Bauer/Kindermann:2008}.
\begin{assumption}\label{assumption:x}
Let $\VarX$ such that either (see \cite{Bauer/Kindermann:2008} eq. (9), slightly
rewritten)
\begin{equation*}
\| (\Op^* \Op)^{-1} \VarX \|_\SpX < \infty
\end{equation*}
or that there exist constants $\gamma > 0,\,\, 1>\nu > 0$,\,\, $C_{\gamma,\nu} >0$ and
$D_{\gamma,\nu} >0$ such that for all $0 \leq t \leq \gamma$  (see
\cite{Bauer/Kindermann:2008} Definition 2.2)
\begin{equation*}
 D_{\gamma,\nu}^2 t^{2 \nu} \geq \sum_{\{ k:| \sigma_k^2 \leq t \}}  \<\VarX,u_k\>^2  \geq C_{\gamma,\nu}^2 t^{2 \nu}
\end{equation*}
\end{assumption}
\begin{remark}
There are other functional analytic formulations of this assumption (cf. \cite{Kindermann/Neubauer:2008}); however, the general idea is similar; namely that $\VarX$ should have a rather uniform distribution of the energy in its coefficients $\< \VarX, \BasX_k \>$.
\end{remark}
\begin{lemma}[cf. \cite{Bauer/Kindermann:2008}]
Let $\VarX$ fulfill assumption \ref{assumption:x}. Then it holds
\begin{equation}\label{impineq}
 \| \VarX - \VarX_{n+1} \|  \leq  w_1 \| \VarX - \VarX_{n} \|
\end{equation}
with $0<w_1 <1$.
\end{lemma}
Now we will introduce two different noise models. Classically, one considers deterministic noise
\begin{definition}[Deterministic Noise]
$\xi$ is called determistic noise of noise level $\delta$ if
\begin{equation*}
\|\xi\| \leq \delta
\end{equation*}
The noise behavior $\rho(\cdot)$ wrt the regularization parameter $n$ is defined as
\begin{equation*}
\rho(n) = \| \Op_n^{-1} \| \delta \leq \delta (q_0 q^n)^{-1/2}
\end{equation*}
\end{definition}
As all the results can be transferred easily to the case of colored noise by modifying the function $\rho$ we will now just consider white noise.
\begin{definition}[Stochastic Noise]
Let $\xi$ a white noise Gaussian random variable, i.e., the  $\<\VarX,\BasX_k\>$ are independent and identically distributed (iid) along the distribution  $\mathcal{N}(0,\delta^2)$.
The noise behavior $\rho(\cdot)$ wrt the regularization parameter $n$ is defined as
\begin{equation*}
\rho(n)^2 = \E \| \Op_n^{-1} \xi \|^2 = \delta^2 \operatorname{trace}(\Op_n^{-1})
\end{equation*}
\end{definition}
In both cases this trivially yields
\begin{lemma}
There exists  $1< w_2$ such that
\begin{equation}\label{impineq2}
\rho(n) \leq \rho(n+1) \leq w_2 \rho(n)
\end{equation}
\end{lemma}
\begin{remark}
Almost every other constant in this article will be based on $w_1$ and $w_2$. Though desirable we cannot give a general upper or lower bound for these constants, they are purely problem dependent. At first glance this might look as a disadvantage, however this means as well that we get problem specific optimal results.
\end{remark}

\section{Fast Balancing}\label{fastbalancing}
The question of the optimality of a regularization parameter $n_{opt}$ is rather difficult in a deterministic setting. The most natural definition, namely
$$n_{o} = \argmin_n \| \VarX_n^\delta - \VarX \|, $$ would be best. However, there is no known concept which leads to proofs in a general setting. Thus we will use the second-best solution. Using the triangle inequality, the error $\| \VarX_n^\delta - \VarX \|$ is bounded by the sum of a decreasing function $\| \VarX_n - \VarX \|$ and an increasing function $\| \Op_n^{-1} \xi\|$ which itself is bounded by $\rho(n)$.
$$ \| \VarX_n^\delta - \VarX \| \leq \| \VarX_n - \VarX \| + \| \Op_n^{-1} \xi\| \leq \| \VarX_n - \VarX \| + \rho(n)$$
Again, the point $$n_{oo} = \argmin_n \| \VarX_n - \VarX \| + \rho(n)$$ is inaccessible. However, using the definition \eref{impineq3} below we can at least guarantee that
\begin{equation*}
2 \(\|\VarX_{n_{oo}} - \VarX\| + \rho(n_{oo})\) \geq  \|\VarX_{n_{opt}} - \VarX\| + \rho(n_{opt}).
\end{equation*}
Interestingly, the stochastic case is much easier due to the independence of $\xi$ and $\VarX$:
\begin{equation*}
\E \| \VarX_n^\delta - \VarX \|^2 = \| \VarX_n - \VarX \|^2 + \rho(n)^2 + 2 \< \(\Op_n^{-1}\)^* (\VarX_n - \VarX) , \xi \> = \| \VarX_n - \VarX \|^2 + \rho(n)^2.
\end{equation*}
In this case, the  parameter $n_{opt}$ as defined below is really optimal on average:
\begin{definition}[Optimal Parameter]
The optimal parameter $n_{opt}$ is defined such that
\begin{equation}\label{impineq3}
\| \VarX_{n_{opt}} - \VarX \|  >  \rho(n_{opt})  \qquad \text{and} \qquad   \| \VarX_{n_{opt}+1} - \VarX \|  \leq  \rho(n_{opt}+1)
\end{equation}
\end{definition}
\begin{remark}
This parameter, of course, just exists when the noise level $\delta$ is sufficiently small. However, if this is not the case the noise described by $\rho(\cdot)$ dominates the information $\VarX$ so much that $0$ would be the best regularized solution. It is common practice to assume that this parameter $n_{opt}$ exists.
\end{remark}
Now we define the Fast and the Lepskij balancing principle and perform a first comparison

\begin{definition}[Balancing Functional]
Let $k \geq 1$. The balancing functional is defined as
\begin{equation}\label{def:balfunc}
b_k(n) = \max_{n< m \leq n + k} \left\{ 4^{-1} \|\VarX_n^\delta - \VarX_m^\delta\|\rho(m)^{-1}  \right\}
\end{equation}
\end{definition}

\begin{definition}[Lepskij Balancing Principle]
Let $b_k(n)$ defined as in \eref{def:balfunc}. Furthermore let $\tau\geq 1$ and $N\geq n_*$. The Lepskij Balancing Parameter $n_L(\tau,N) = n_L$ is defined as
\begin{equation}\label{def:lepskibal}
n_L(\tau,N) = \argmin_{n} \left\{ b_k(m) < \tau \quad \forall N \geq m \geq n\right\}
\end{equation}
\end{definition}

\begin{definition}[Fast Balancing Principle]
Let $b_k(n)$ defined as in \eref{def:balfunc} and let $\tau\geq 0$. The fast balancing parameter\\ $n_*=n_*(\tau)$ is defined as
\begin{equation}\label{def:fastbal}
n_*(\tau) = \argmin_{n} \left\{ b_k(n) < \tau  \right\}
\end{equation}
\end{definition}
\begin{remark}
In contrast to the Lepskij balancing principle \cite{Mathe/Pereverzev:2003,Mathe/Pereverzev:2006,Bauer/Munk:2007}, no upper bound $N$ is needed.
\end{remark}
\begin{lemma}
It holds for all admissible pairs $(N,\tau)$
\begin{equation*}
n_L(\tau,N) \geq \min \{N,n_*(\tau)\}
\end{equation*}
\end{lemma}
\begin{proof}
This is a direct consequence out of \eref{def:lepskibal} and \eref{def:fastbal}.
\end{proof}
\begin{remark}
For the sake of simpler notation we will mostly refer to $n_L(\tau,N)$ by $n_L$ and to $n_*(\tau)$ by $n_*$. Just when the particular choice of $N$ and $\tau$ is importance we will keep these parameters.

The condition that $N\geq n_*$ is of course rigorously seen not fulfillable without knowing the optimal regularization parameter, however this is a standard assumption for all proofs for the Lepskij balancing principle and does not pose any particular problems in practice where $N$ is normally chosen anyway in the range of the machine precision.
\end{remark}

\section{Deterministic Case}
\label{deterministic}
Now we will show that -- as for the Lepskij balancing principle -- for the fast balancing principle an oracle inequality holds.
\begin{theorem}
Let $n_*$ and $n_{opt}$ defined as above. Furthermore assume assumption \ref{assumption:x} to be valid. It holds for the constants $w_1$, $w_2$ defined in \eref{impineq} and \eref{impineq2} independent of $\delta$
\begin{equation*}
\|\VarX_{n_*}^\delta - \VarX\| \leq C \(\|\VarX_{n_{opt}} - \VarX\| + \rho(n_{opt}) \)
\end{equation*}
where $C = \frac{(4\tau +3)}{2}\min_{1\leq \kappa \leq k} \left\{\frac{w_2^\kappa}{1 -
w_1^{\kappa}}\right\}$.
\end{theorem}

\begin{proof}
Assume $m > n > n_{opt}$. Due to assumption \ref{assumption:x} and either definition of the noise the equations \eref{impineq} and \eref{impineq2} hold and we obtain with the triangle inequality
\begin{align*}
\|\VarX_{n}^\delta - \VarX_{m}^\delta\| &\leq \|\VarX_{n} - \VarX\| +\|\VarX_{m} - \VarX\| + \rho(n) + \rho(m) <4 \rho(m)
\end{align*}
and hence
\begin{equation*}
b_k(n) < 1
\end{equation*}
and thus $n_* \leq n_{opt}$ and hence $\rho(n_*)\leq \rho(n_{opt})$ due to \eref{impineq2}. Using \eref{impineq3} we have
\begin{equation*}
\|\VarX_{n_*} - \VarX\| \geq \rho(n_*).
\end{equation*}
On the other hand, for $n \leq n_*$ and $\kappa \leq k$, using the inverse triangle
inequality and \eref{impineq} resp. \eref{impineq2}:
\begin{align*}
4 \tau \rho(n+\kappa) \geq & \|\VarX_n^\delta - \VarX_{n+\kappa}^\delta\| \\
\geq & \|\VarX_n - \VarX\| - \|\VarX_{n+\kappa} - \VarX\| - \rho(n) - \rho(n+\kappa) \\
\geq & (1 - w_1^{\kappa}) \|\VarX_n - \VarX\| - 2 \rho(n+\kappa).
\end{align*}
Hence
\begin{equation*}
\|\VarX_{n_*} - \VarX\| \leq (4\tau +2)\min_{1\leq \kappa \leq k} \left\{\frac{w_2^\kappa}{1 -
w_1^{\kappa}}\right\} \rho(n_*)
\end{equation*}
and so using \eref{impineq2}
\begin{align*}
\|\VarX_{n_*}^\delta - \VarX\| &\leq \|\VarX_{n_*} - \VarX\| + \rho(n_*) \\
 &\leq (4\tau +3)\min_{1\leq \kappa \leq k} \left\{\frac{w_2^\kappa}{1 - w_1^{\kappa}}\right\} \rho(n_*) \\
 &\leq (4\tau +3)\min_{1\leq \kappa \leq k} \left\{\frac{w_2^\kappa}{1 - w_1^{\kappa}}\right\} \rho(n_{opt}) \\
 &\leq \frac{(4\tau +3)}{2}\min_{1\leq \kappa \leq k} \left\{\frac{w_2^\kappa}{1 - w_1^{\kappa}}\right\} \(\|\VarX_{n_{opt}} - \VarX\| + \rho(n_{opt}) \).
\end{align*}
\end{proof}
\begin{remark}
Obviously, we obtain the best result for the minimal admissible $\tau$, i.e. $\tau = 1$. Furthermore, $C$ is independent of $\delta$ as $w_1$ and $w_2$ are.
\end{remark}

\section{Stochastic Case}
\label{stochastic}

We will need a stochastic bound for the probability that the observed error is much bigger than our estimation $\rho(\cdot)$.

\begin{lemma}[see e.g., \cite{Bauer/Reiss:2008}]\label{ProbLemma}
Let $Z=\sum_{k=1}^\infty\alpha_k^2\zeta_k^2$ with $\sum_{k=1}^\infty\alpha_k^2=1$ and
$\zeta_k\sim N(0,1)$ iid. Assume that $\max_k \alpha_k >0$. Then
\begin{equation}\label{sto}
\forall\,z>0:\;\mathbb{P}(Z\ge z)\le \sqrt{2}e^{-z/4}.
\end{equation}
\end{lemma}
The proofs would become far too complicated if we used $b_k(\cdot)$ for any $k$. Therefore, we will restrict our attention to the case $k=1$, i.e., $b_1(\cdot)$. In numerical implementations we observed that the results improve for slightly bigger $k$. From a certain $k$ onwards (problem-dependent) it turns out that $k$ does not seem to have an influence on the solution any more.
\begin{lemma}\label{problemma2}
Assume assumption \ref{assumption:x} to be valid.
For $n<n_{opt}$ there exist constants $c_1, c_2 > 0$ depending on $w_1$ and $w_2$ but not depending on $\delta$ and $\tau$ such that
\begin{equation}\label{probineq}
\PP \left\{b_1(n) < \tau \right\} \leq c_1 \exp\(-c_2(n_{opt}-n)^2\) \exp(\tau^2)
\end{equation}
Furthermore for $n > n_{opt}$ there exists a constant  $c_3 >0$ depending on $w_1$ and $w_2$ but not depending on $\delta$ and not depending on $\tau$ such that
\begin{equation}\label{probineq2}
\PP \left\{b_1(n)
\geq \tau \right\} \leq c_3 \exp\(-\tau^2/4\).
\end{equation}
\end{lemma}
\begin{proof}
Let $n<n_{opt}$. Then using \eref{eq1}, \eref{eq2}, \eref{eq3}, \eref{impineq}, \eref{impineq2}, \eref{impineq3},\eref{sto} and the triangle inequality we get  (usage of the equations marked on top)
\begin{align*}
\PP &\left\{4^{-1} \|\VarX_n^\delta - \VarX_{n+1}^\delta\|\rho(n+1)^{-1} < \tau \right\}
\\& = \PP \left\{4 \tau \rho(n+1) > \|\VarX_n^\delta - \VarX_{n+1}^\delta\| \right\}
\\& \overset{ti}{\leq} \PP \left\{4 \tau\rho(n+1) > \|\VarX_n - \VarX\| - \|\VarX_{n+1} - \VarX\| - \|\VarX_n^\delta - \VarX_n\| - \|\VarX_{n+1}^\delta - \VarX_{n+1}\| \right\}
\\& \overset{\eref{impineq}}\leq \PP \left\{4\tau  > \frac{(1-w_1) \|\VarX_n - \VarX\|}{\rho(n+1)}  - \frac{\|\VarX_n^\delta - \VarX_n\|}{\rho(n+1)} - \frac{\|\VarX_{n+1}^\delta - \VarX_{n+1}\|}{\rho(n+1)} \right\}
\\& = \PP \left\{ \frac{\|\VarX_n^\delta - \VarX_n\|}{\rho(n+1)} + \frac{\|\VarX_{n+1}^\delta - \VarX_{n+1}\|}{\rho(n+1)} > \frac{(1-w_1) \|\VarX_n - \VarX\|}{\rho(n+1)}  -4\tau \right\}
\\& \overset{\eref{impineq}}{\leq} \PP \left\{\frac{\|\VarX_n^\delta - \VarX_n\|}{\rho(n+1)} + \frac{\|\VarX_{n+1}^\delta - \VarX_{n+1}\|}{\rho(n+1)} > \frac{(1-w_1) w_1^{-(n_{opt}-n)} \|\VarX_{n_{opt}} - \VarX\|}{\rho(n_{opt})} - 4\tau  \right\}
\\& \overset{\eref{eq1}\eref{eq2}\eref{eq3}\eref{impineq2}\eref{impineq3}}{\leq} \PP \left\{\frac{\|A_n^{-1} \xi\|}{\rho(n)} + \frac{\|A_{n+1}^{-1} \xi\|}{\rho(n+1)} > {(1-w_1) w_1^{-(n_{opt}-n)}} - 4\tau  \right\}
\\& \leq \PP \left\{\frac{\|A_n^{-1} \xi\|}{\rho(n)}  > {\frac{1-w_1}{2} w_1^{-(n_{opt}-n)}} - 2\tau
\right\} + \PP \left\{ \frac{\|A_{n+1}^{-1} \xi\|}{\rho(n+1)} > {\frac{1-w_1}{2}
w_1^{-(n_{opt}-n)}} -2\tau \right\} \\
& \overset{\eref{sto}}{\leq} 2\sqrt{2} \exp\(-\(\frac{1-w_1}{2} w_1^{-(n_{opt}-n)} -2\tau\)^2/4\) \\
& \leq 2\sqrt{2} \exp\(-\(\frac{1-w_1}{4} w_1^{-(n_{opt}-n)}\)^2\) \exp(\tau^2) \\
& \leq c_1 \exp\(-c_2(n_{opt}-n)^2\) \exp(\tau^2)
\end{align*}
for some appropriate constants $c_1$ and $c_2$. Now let $n > n_{opt}$. Then using \eref{eq1}, \eref{eq2}, \eref{eq3},  \eref{impineq2}, \eref{impineq3}, \eref{sto} and the triangle inequality we get  
\begin{align*}
 \PP &\left\{4^{-1} \|\VarX_n^\delta - \VarX_{n+1}^\delta\|\rho(n+1)^{-1}
\geq \tau \right\} \\
& = \PP \left\{4\tau \rho(n+1) \leq \|\VarX_n^\delta - \VarX_{n+1}^\delta\| \right\}\\
& \overset{ti}{\leq} \PP \left\{4\tau \rho(n+1) \leq \|\VarX_n - \VarX\| + \|\VarX_{n+1} + \VarX\| +
\|\VarX_n^\delta - \VarX_n\| + \|\VarX_{n+1}^\delta - \VarX_{n+1}\| \right\}\\
& \overset{\eref{impineq2}\eref{impineq3}}{\leq} \PP \left\{(4\tau -2 ) \rho(n+1) \leq \|\VarX_n^\delta - \VarX_n\| +
\|\VarX_{n+1}^\delta - \VarX_{n+1}\| \right\}\\
& \overset{\eref{eq1}\eref{eq2}\eref{eq3}\eref{impineq2}}{\leq} \PP \left\{2\tau \leq \frac{\|A_n^{-1} \xi\|}{\rho(n)} + \frac{\|A_{n+1}^{-1} \xi\|}{\rho(n+1)} \right\}\\
& \leq \PP \left\{\tau \leq \frac{\|A_n^{-1} \xi\|}{\rho(n)}\right\} + \PP \left\{\tau\leq \frac{\|A_{n+1}^{-1} \xi\|}{\rho(n+1)} \right\}\\
& \overset{\eref{sto}}{\leq} 2\sqrt{2} \exp\(-\tau^2/4\) .
\end{align*}
\end{proof}

\begin{lemma} Assume assumption \ref{assumption:x} to be valid and constants $c_1$,$c_2$ as defined in lemma \ref{problemma2}.
Let $k=1$. Then for $n < n_{opt}$
\begin{equation}
\PP \left\{ n_* = n \right\} \leq c_1 \exp\(-c_2(n_{opt}-n)^2\) \exp(\tau^2).
\end{equation}
\end{lemma}
\begin{proof}
It holds:
\begin{align*}
\PP \left\{ n_* = n \right\} &= \PP \left\{ b_1(n) < \tau \quad \text{and} \quad \forall_{m<n} b_1(n)>\tau \right\}\\
&\leq \PP \left\{ b_1(n) < \tau \right\} \leq c_1 \exp\(-c_2(n_{opt}-n)^2\) \exp(\tau^2).
\end{align*}
\end{proof}
The situation for $n>n_{opt}$ becomes much more complicated. It all depends on the question how fast the probabilities $\PP \left\{b_1(n)
\geq \tau \right\}$ and $\PP \left\{b_1(m)
\geq \tau \right\}$ decorrelate. Therefore, we will first present both the extreme cases and then discuss their implications.
\begin{lemma}
Assume assumption \ref{assumption:x} to be valid and $c_3$ defined as in lemma \ref{problemma2}.
Let $n>n_{opt}$. Then
\begin{equation}
\PP \left\{ n_* \in\{n_{opt}+1 \ldots n \} \right\} \leq c_3 \exp\(-\tau^2/4\).
\end{equation}
Assume additionally that for all $m<n$ it holds $\PP \left\{b_1(n) \geq \tau   \quad \text{and} \quad b_1(m) \geq \tau\right\} = \PP \left\{b_1(n)\geq \tau \right\} \PP \left\{b_1(m)\geq \tau \right\}$. Then
\begin{equation}
\PP \left\{ n_* = n \right\} \leq \( c_3 \exp\(-\tau^2/4\) \)^{n-n_{opt}}.
\end{equation}
\end{lemma}
\begin{proof}
Direct consequence of lemma \ref{problemma2}.
\end{proof}
In the case of decorrelation, the likelihood of the event $n_* > n_{opt}$ decreases very fast, whereas in the worst case (i.e. the perfectly correlated case) the likelihood is constant. Please note that in any case we are in a considerably better situation than with  the Lepskij balancing principle, where, in dependence of an upper bound $N$, the probability is $N \exp(-\tau^2/4)$.

In reality we cannot expect complete decorrelation of events, however
considering practical observations the following relaxed condition seems to be reasonable:
\begin{assumption}\label{assumption:tau}
There exists a constant $C_3(w_1,w_2) > c_3$ such that for $n > n_*$
\begin{equation}\label{def:decor}
\PP \left\{ n_* = n \right\} \leq \( C_3 \exp\(-\tau^2/4\) \)^{n-n_{opt}}.
\end{equation}
Assume furthermore that $\tau$ big enough such that
\begin{equation}\label{def:decor2}
C_3 \exp\(-\tau^2/4\) w_2 < 1
\end{equation}
\end{assumption}
\begin{remark}
Obviously we have no justified way to find out which $\tau$ fulfills \eref{def:decor2}; however as $C_3$ just depends on $w_1$ and $w_2$ and not on $\delta$ we also have this independence from $\delta$ for $\tau$.

In most practical situations $\tau =1$ seems to be sufficient and a dependence of $\delta$ has not been observed (which might of course just be due to the limited range of numbers processable on modern computer hardware).
\end{remark}

\begin{theorem}
Assume that assumptions \ref{assumption:x} and \ref{assumption:tau} are valid.
Then we have
\begin{equation*}
\E \|\VarX_{n_*}^\delta - \VarX\|^2 \leq c_5 \( 1 +  c_6 \exp(\tau^2/2)  + c_5\frac{1}{1 - C_3
\exp\(-\tau^2/8\) w_2 } \) \( \|\VarX_{n_{opt}} - \VarX\|^2 + \rho(n_{opt})^2 \)
\end{equation*}
with variables $c_5$ and $c_6$ just depending on $w_1$ and $w_2$.
\end{theorem}
\begin{proof}
Due to the independence of $\VarX$ and $\xi$ and an inequality connecting the different moments of Gaussian random variables (see e.g., \cite{Bauer/Reiss:2008}), we have for $c_5 = (4 \Gamma(3))^{1/4}$
\begin{align*}
\E \| \VarX - \VarX^\delta_n \|^4 &=
\| \VarX - \VarX_n \|^4 +
2 \| \VarX - \VarX_n \|^2\, \E \| \VarX_n - \VarX^\delta_n \|^2 +
\E \| \VarX_n - \VarX^\delta_n \|^4 \\
& \leq \| \VarX - \VarX_n \|^4 +
2 \| \VarX - \VarX_n \|^2\, \E \| \VarX_n - \VarX^\delta_n \|^2 +
(c_5\, \E \| \VarX_n - \VarX^\delta_n \|^2)^2 \\
& \leq \( \| \VarX - \VarX_n \|^2 + c_5\, \E \| \VarX_n - \VarX^\delta_n \|^2 \)^2 \\
& \leq c_5^2 \( \| \VarX - \VarX_n \|^2 + \E \| \VarX_n - \VarX^\delta_n \|^2 \)^2 \\
& = c_5^2 \(\E \| \VarX - \VarX^\delta_n \|^2\)^2.
\end{align*}
For an appropriate constant $c_6$ independent of $\tau$ it holds using \eref{impineq}, \eref{impineq2}, \eref{impineq3}, \eref{probineq}, \eref{def:decor} and the H\"older inequality:
\begin{align*}
\E \|\VarX_{n_*}^\delta - \VarX\|^2
=&  \sum_{n=1}^\infty \E \(\| \VarX - \VarX^\delta_n \|^2 \mathbf{1}_{n=n_*}\) \\
\leq &  \sum_{n=1}^{n_{opt}-1} \(\E \| \VarX - \VarX^\delta_n \|^4\)^{1/2} \(\E
\mathbf{1}_{n=n_*}^{2}\)^{1/2} \\
&\qquad+ \E \|\VarX_{n_{opt}}^\delta - \VarX\|^2 \\
&\qquad+ \sum_{n=n_{opt}+1}^{\infty} \(\E \| \VarX - \VarX^\delta_n \|^4\)^{1/2} \(\E
\mathbf{1}_{n=n_*}^{2}\)^{1/2} \\
\overset{\eref{probineq}\eref{def:decor}}{\leq} & \sum_{n=1}^{n_{opt}-1} c_5 \( \| \VarX - \VarX_n \|^2 +  \rho(n)^2 \)\(c_1 \exp\(-c_2(n_{opt}-n)^2/2\) \exp(\tau^2/2)\)\\
&\qquad+ \| \VarX - \VarX_{n_{opt}} \|^2 + \rho(n_{opt})^2 \\
&\qquad+ \sum_{n=n_{opt}+1}^{\infty}c_5 \(\| \VarX - \VarX_{n} \|^2 +  \rho(n)^2\) \(C_3 \exp\(-\tau^2/8\)\)^{n-n_{opt}}\\
\overset{\eref{impineq}\eref{impineq2}\eref{impineq3}}\leq & \sum_{n=1}^{n_{opt}-1} c_5 \( w_1^{-(n_{opt}-n)}\| \VarX - \VarX_{n_{opt}} \|^2 +  \rho(n_{opt})^2 \)\(c_1 \exp\(-c_2(n_{opt}-n)^2/2\) \exp(\tau^2/2)\)\\
&\qquad+ \| \VarX - \VarX_{n_{opt}} \|^2 + \rho(n_{opt})^2 \\
&\qquad+ \sum_{n=n_{opt}+1}^{\infty}c_5 \(\| \VarX - \VarX_{n_{opt}} \|^2 +   w_2^{n-n_{opt}} \rho(n_{opt})^2\) \(C_3 \exp\(-\tau^2/8\)\)^{n-n_{opt}}\\
\leq &  \| \VarX - \VarX_{n_{opt}} \|^2 c_5 \(1 + c_6 \exp(\tau^2/2) + \frac{1}{1 - C_3\exp\(-\tau^2/8\) } \) \\
&\qquad+ \rho(n_{opt})^2 c_5 \( 1 +  c_6 \exp(\tau^2/2)  + c_5 \frac{1}{1 - C_3
\exp\(-\tau^2/8\) w_2 } \),
\end{align*}
which proves the assertion.
\end{proof}
Even if we cannot assume decorrelation (assumption \ref{assumption:tau}), we still have a convergence
result due to $n_L(\tau) \geq \min \{N,n_*(\tau)\}$, also compare \cite{Raus/Haemarik:2008}.
\begin{theorem}Assume that assumptions \ref{assumption:x} to be valid.
Let $n_* = \min \{N,n_*(\tau)\}$ and $\tau = c \log(\delta^-1)$. Then it holds
\begin{equation*}
\E \|\VarX_{n_L}^\delta - \VarX\|^2 \leq C_L  \( \|\VarX_{n_{opt}} - \VarX\|^2 +
\log(\delta^{-1}) \rho(n_{opt})^2 \)
\end{equation*}
and
\begin{equation*}
\E \|\VarX_{n_*}^\delta - \VarX\|^2 \leq C_*  \( \|\VarX_{n_{opt}} - \VarX\|^2 +
\log(\delta^{-1}) \rho(n_{opt})^2 \).
\end{equation*}
\end{theorem}
\begin{proof}
The first part was proven in \cite{Bauer/Pereverzev:2005}; the second part is
trivial, using that on the one hand $n_L(\tau) \geq \min \{N,n_*(\tau)\}$. On the
other hand, the risk for $n_*$ being too small (which is not affected by the decorrelation effect) can be bounded from above by a multiple of
$\|\VarX_{n_{opt}} - \VarX\|^2 + \rho(n_{opt})^2$ as shown in the last theorem.
\end{proof}

\section{Conclusion}

In numerical experiments \cite{Bauer/Lukas:2010} it is almost impossible to distinguish the results of the
Lepskij balancing principle and the newly introduced fast balancing. This can be
interpreted as follows:
\begin{itemize}
\item There is just a very low probability for outliers in the noise which influence the noise behavior $\rho(\cdot)$ after the optimal regularization parameter.
\item There is an extremely low probability that the noise modifies the data in such a way that one stops too early.
\end{itemize}
However, considering the computation time, the new method has big advantages, it can
compete easily with other methods like the Morozov discrepancy principle. 

In conclusion, we have shown that this modification of the Lepskij balancing principle is
very well suited for practice and should replace it in all time-critical applications
as long as one does not need to fear big frequency gaps in the solutions.

This analysis has been done only for Tikhonov regularization. A similar analysis imposing stricter requirements on the solution $\VarX$ was performed for truncated Singular Value Decomposition (TSVD) in \cite{Bauer:2009}.

Furthermore, large numerical experiments \cite{Bauer/Lukas:2010} show that the newly defined method works very
well and, in contrast to most other parameter choice regimes, can cope with colored
noise without any performance loss. In these experiments it was observed that the factor
$C$ in the oracle inequality is at most around $2$. The method is very stable, i.e., the
number of observed outliers is very low, both for Tikhonov and Spectral-Cut-Off
regularization. This behavior does not change when one replaces the exact noise behavior
by an estimation based on several measurements \cite{Bauer:2009}.

\section*{Acknowledgements}

The author gratefully acknowledges the financial support by the Upper
Austrian Technology and Research Promotion.

\section*{References}
\bibliographystyle{amsalpha}
\bibliography{bibliography}

\end{document}